\documentclass[11pt]{article}

\usepackage{amsmath}
\usepackage{amsfonts}
\usepackage{changepage}
\usepackage{color}
\usepackage{latexsym}
\usepackage{tablefootnote}
\usepackage{epsfig}
\usepackage{multirow}
\usepackage{float}
\usepackage{array}
\usepackage{bbm}

\allowdisplaybreaks

\newtheorem{thm}{Theorem}[section]
\newtheorem{theorem}[thm]{Theorem}

\newtheorem{lemma}[thm]{Lemma}

\newtheorem{proposition}[thm]{Proposition}

\newtheorem{remark}[thm]{Remark}

\newcommand{\lam}{{\lambda}}

\newcommand{\argmin}{\mathrm{argmin}\,}

\newcommand{\dom}{\mathrm{dom}\,}

\newcommand{\tx}{\tilde x}
\newcommand{\ty}{\tilde y}

\newcommand{\mb}[1]{\mbox{\boldmath $#1$}}

\newenvironment{myenv}{\begin{adjustwidth}{1.5cm}{}}{\end{adjustwidth}}

\begin{document}
\title{A FISTA-Type First Order Algorithm on Composite Optimization Problems that is Adaptable to the Convex Situation}

\author{{\bf{Chee-Khian Sim}}}%\footnote{Email address:
%macksim@inet.polyu.edu.hk} \\ Department of Applied Mathematics\\
%The Hong Kong Polytechnic University \\ Hung Hom, Kowloon \\ Hong
%Kong}

\date{May 25, 2020}
\maketitle
\begin{abstract}
In this note, we propose a FISTA-type first order algorithm, VAR-FISTA, to solve a composite optimization problem.  A distinctive feature of VAR-FISTA is its ability to exploit the convexity of the function in the problem, resulting in an improved iteration complexity when the function is convex compared to when it is nonconvex.  The iteration complexity result for the convex and nonconvex case obtained in the note are compatible to the best known in the literature so far. %An iteration bound result for VAR-FISTA to find an approximate solution to the nonconvex composite program is shown in the note.  The analysis to arrive at the result follows \cite{Liang} closely.

\vspace{15pt}

\noindent {\bf{Keywords.}} Fast iterative shrinkage thresholding algorithm (FISTA); Composite optimization problem; Iteration complexity.
\end{abstract}

\section{Introduction}\label{sec:Intro}

Using first order methods is the preferred approach to solve large scale optimization problems that arise in application areas such as machine learning.  Fast iterative shrinkage thresholding algorithm (FISTA), an efficient first order method, is proposed in \cite{beck2009fast} to solve composite optimization problems when the functions involved are convex; see also \cite{ag_nesterov83,nesterov2012gradient,Nest05-1}.  Recently, there are interests in the study of first order algorithms, such as FISTA variants, to solve composite optimization problems with nonconvex functions.  These works include \cite{carmon2018accelerated,nonconv_lan16, GhadimiLanZhang,KongMeloMonteiro,Li_Lin2015,
Lin_Zhou_Liang_Varshney, jliang2018double, Liang2,Liang,tseng2008accmet, Yao_et.al.}. % Furthermore, in \cite{carmon2019}, an explicit lower bound to iteration complexity using any first order method to solve nonconvex composite optimization problems is provided.

\vspace{10pt}

\noindent In this note, we propose a FISTA-type first order algorithm, VAR-FISTA, to solve composite optimization problems.  This algorithm is inspired by the algorithm ADAP-NC-FISTA in \cite{Liang}.  The algorithm  in this note is designed in such a way that when the functions involved in the composite optimization problem are convex, it is able to exploit the convexity of the problem leading to an iteration complexity of $\mathcal{O}((1/\hat{\rho})^{2/3})$, while an iteration complexity of $\mathcal{O}(1/\hat{\rho}^2)$ is achieved in the nonconvex case.  These complexity results are the best known in the literature so far.  The contribution of this note is twofold.  First, the algorithm requires only one resolvent evaluation in an iteration, unlike the algorithms in \cite{GhadimiLanZhang,Li_Lin2015}.  Second, other than information on function and gradient values, no other data  information, such as Lipschitz constant or lower curvature, are required from the problem, as in \cite{nonconv_lan16,Yao_et.al.}, for the algorithm to run.  It should finally be  noted that the algorithm, ADAP-NC-FISTA, in \cite{Liang} also shares these same features as our algorithm, but in \cite{Liang}, the iteration complexity of 
$\mathcal{O}((1/\hat{\rho})^{2/3})$ when the functions in the composite optimization problem are convex cannot be directly established.

%\cite{ag_nesterov83}; \cite{nonconv_lan16}; \cite{carmon2019}; \cite{tseng2008accmet}; \cite{jliang2018double}; \cite{Liang2}; \cite{Yao_et.al.}; \cite{Lin_Zhou_Liang_Varshney}; \cite{Li_Lin2015}; \cite{carmon2018accelerated}; \cite{KongMeloMonteiro}; \cite{Nest05-1}; \cite{nesterov2012gradient}; \cite{GhadimiLanZhang}

\section{A Composite Optimization Problem}\label{sec:CNO} 

We consider the following composite optimization problem:
\begin{eqnarray}\label{CNO}
\min_{u \in \Re^n} \phi(u):= f(u) + h(u),
\end{eqnarray}
where $f$ is continuously differentiable, can be nonconvex on $\Omega$, 
\begin{eqnarray}\label{lipschtiz}
\| \nabla f(u_1) - \nabla f(u_2) \| \leq M \|u_1 - u_2\|,\ \forall\ u_1, u_2 \in \Omega,
\end{eqnarray}
with $M > 0$ and $\Omega$ is a closed convex set in $\Re^n$, that is, the gradient of $f$ is Lipschitz continuous on $\Omega$, and $h$ is a proper lower semi-continuous convex function, which can be nonsmooth, with ${\mbox{dom}}\ h \subset \Re^n$ being closed and bounded.  We assume that ${\mbox{dom}}\ h \subseteq \Omega$.  Let $\overline{M} (> 0)$ be the smallest $M$ satisfying (\ref{lipschtiz}).  There exists $m \geq 0$ such that
\begin{eqnarray}\label{lowercurvature}
-\frac{m}{2} \| u_1 - u_2 \|^2 \leq f(u_1) - \ell_f(u_1;u_2), \forall\ u_1, u_2 \in \Omega,
\end{eqnarray}
where
\begin{eqnarray}\label{linearization}
\ell_f(u_1;u_2) := f(u_2) + \langle \nabla f(u_2), u_1 - u_2 \rangle,
\end{eqnarray}
since by (\ref{lipschtiz}),
\begin{eqnarray}\label{ineq:lipschitzconsequence}
| f(u_1) - l_f(u_1; u_2) | \leq \frac{M}{2} \| u_1 - u_2 \|, \ \forall\ u_1, u_2 \in \Omega.
\end{eqnarray}
Hence, an $m$ that satisfies (\ref{lowercurvature}) is $\overline{M}$.  Let $\underline{m} \geq 0$ be the smallest $m \geq 0$ such that (\ref{lowercurvature}) holds.  We have $0 \leq \underline{m} \leq \overline{M}$.  Observe that if $\underline{m} = 0$, then by (\ref{lowercurvature}), $f$ is convex on $\Omega$, while if $\underline{m} > 0$, then $f$ is nonconvex on $\Omega$.  % Also, since $f$ satisfies (\ref{lipschtiz}), we have
%\begin{eqnarray}\label{uppercurvature}
%f(u_1) - \ell_f(u_1;u_2) \leq \frac{M}{2} \| u_1 - u_2 \|^2, \forall\ u_1, u_2 \in \Omega.
%\end{eqnarray}

\vspace{10pt}

\noindent Let us denote $y^\ast \in {\rm{dom}}\ h$ to be an optimal solution to Problem (\ref{CNO}), which exists since ${\rm{dom}}\ h$ is closed and bounded.

\vspace{10pt}

\noindent A necessary condition for ${u} \in {\mbox{dom}}\ h$ to be a local minimum of Problem (\ref{CNO}) is $0 \in \nabla f({u}) + \partial h({u})$.  Motivated by this condition, we have the notion of an $\hat{\rho}$-approximate solution of Problem (\ref{CNO}), which is a pair $(\hat{u}, \hat{v})$ which satisfies
\begin{eqnarray}\label{approximatesolution}
\hat{v} \in \nabla f(\hat{u}) + \partial h(\hat{u}),\quad \| \hat{v} \| \leq \hat{\rho},
\end{eqnarray}
where $\hat{\rho} > 0$ is a given tolerance.

\section{A FISTA-Type Algorithm on Problem (\ref{CNO})}\label{sec:FISTA}

\noindent In the following, we propose a variant of FISTA, which we call VAR-FISTA, to find an $\hat{\rho}$-approximate solution to Problem (\ref{CNO}), for a given tolerence $\hat{\rho} > 0$.  This algorithm is inspired by the algorithm ADAP-FISTA in \cite{Liang}.

\noindent\rule[0.5ex]{1\columnwidth}{1pt}

{\centerline{\bf{{VAR-FISTA}}}}

\noindent\rule[0.5ex]{1\columnwidth}{1pt}

\noindent {\bf{Initialization}}: Let $\xi_0=0$, $\lam_0>0$, $\theta> 1$, $0 < \gamma < 1$, tolerance $\hat \rho > 0$ and 
	initial point $y_0 \in \dom h$, and set $y_0^{\rm{min}}=x_0=y_0$, $A_0 = 12$, $L_{0} = 0$.
	
\vspace{10pt}

\noindent {\bf{$\mb{k^{\rm{th}}}$ Iteration}} ($k = 1, 2, \ldots$): 
\begin{itemize}
	\item[$k1$.] Set $\lam=\lam_{k-1}$, $ \xi=\xi_{k-1}$  and compute	
	\begin{equation}\label{eq:all ak}
	a_{k-1} = \frac{1+ \sqrt{1 + 4 A_{k-1}}}{2},\quad
	A_{k} = A_{k-1} + a_{k-1}, \quad
	\tx_k  =  \frac{A_{k-1}}{A_{k}}y_{k-1} + \frac{a_{k-1}}{A_{k}}x_{k-1};
	\end{equation}
	%{\color{black}{set $\lambda = \lambda_k$;}} 
	\item[$k2$.] compute
	\begin{align}
	& \tau = \frac{2\xi \lam}{a_{k-1}}, \nonumber\\
	& {y} = {\mbox{argmin}}_u \left\{ l_f(u; \tilde{x}_k) + h(u) + \frac{1 + \tau}{2 \lambda} \| u - \tilde{x}_k \|^2 \right\}, \label{def:yk} \\
    & {\color{black}{U}} {\color{black}{=  \frac{2 \left[   f(y) - \ell_f(y; \tx_k) \right]} {\|y-\tx_k\|^2}; }}\label{def:U} \\
    & \tilde{y}^{\min} = \argmin \left \{  \phi(\ty) \ ;\ \ty = y_{k-1}^{\rm{min}},y  \right\}, \label{eq:ymin} \\
	& L = \max\left\lbrace \frac{2 [\ell_f(y_{k-1}; \tx_k)-f(y_{k-1}) ]} {\|y_{k-1}-\tx_k\|^2},\frac{2 [\ell_f(\tilde{y}^{\min}; \tx_i)-f(\tilde{y}^{\min}) ]} {\|\tilde{y}^{\min}-\tx_i\|^2}, L_{k-1}, 0\ ; \ i = 1, \ldots, k \right\rbrace; \label{def:L}
%	\eta &=  \frac{\|y-\tx_k\|^2} {2| f(y) - \ell_f(y; \tx_k) |},\nonumber
	\end{align}
	\item[$k3$.] If $U \lambda > \gamma$ or $\xi \lambda_{k-1} < L \lambda + \tau$ or $\xi \lambda_{i-1} < L \lambda_i + \tau_i$ for some $i = 1, \ldots, k-1$, go to step $k2$ with $(\xi, \lambda)$ given by
\begin{eqnarray*}	
	(\xi, \lambda) = {\rm{UPDATE}}(\xi, \lambda, \lambda_1, \ldots, \lambda_{k-1}, \tau, \tau_1, \ldots, \tau_{k-1}, L, U, \theta, \gamma);
\end{eqnarray*}
\noindent else set
	$ \tau_{k}=\tau $, $y_{k}^{\min} = \tilde{y}^{\min}$, $y_{k}=y$, $ \lam_{k}=\lam $, $U_k = U$, $ L_k=L $ and $ \xi_{k}=\xi $;
	 \item[$k4$.] compute 
	\begin{align}
	& x_{k}= P_\Omega\left(  \frac{(1 + \tau_k)A_{k}}{a_{k-1}(\tau_k a_{k-1} + 1)} y_{k} - \frac{A_{k-1}}{a_{k-1}(\tau_k a_{k-1} + 1)}y_{k-1} \right), \label{def:xk} \\
	& v_{k} =  \frac{1+\tau_k}{\lam_{k}} (\tilde{x}_k - y_{k}) + \nabla f(y_{k}) - \nabla f(\tilde{x}_k). \label{def:vk+1}
	\end{align}
\end{itemize}

\noindent {\bf{Termination}}: If at the end of the $k^{th}$ iteration, $\|v_{k} \| \le \hat \rho$, then output  $(\hat y,\hat v)=(y_{k},v_{k})$, and {\bf{exit}}.

\noindent \rule[0.5ex]{1\columnwidth}{1pt}

\noindent We describe the subroutine in VAR-FISTA in the following:

\noindent\rule[0.5ex]{1\columnwidth}{1pt}

\noindent {$\mb{(\xi, \lambda) = {\rm{UPDATE}}(\xi, \lambda, \lambda_1, \ldots, \lambda_{k-1}, \tau, \tau_1, \ldots, \tau_{k-1}, L, U, \theta, \gamma)}$}:

\vspace{10pt}

\noindent if $U \lam > \gamma$ then set
	\begin{equation}\label{eq:lamupdate}
	\lambda \leftarrow  \min\{ \lam/\theta, {\color{black}{\gamma/U}} \};
	\end{equation}
	if 
	 \begin{equation}\label{ineq:check}
	 \xi \lam_{k-1}< L \lam +\tau\ {\mbox{or}}\ \xi \lambda_{i-1} < L \lambda_{i} + \tau_i\ {\mbox{for\ some}}\ i = 1, \ldots, k-1
	 \end{equation}
	 then 
	 \begin{myenv}
	 if $\xi = 0$, set
	 \begin{equation}\label{eq:xiupdate0}
	 \xi \leftarrow 1;
	 \end{equation}
	 else set
	 \begin{equation}\label{eq:xiupdate}  
	\xi \leftarrow 2 \xi;
	\end{equation}
	\end{myenv}

\noindent \rule[0.5ex]{1\columnwidth}{1pt}

\noindent In the above algorithm, VAR-FISTA, steps $k2$ and $k3$ can be performed more than once in an iteration since if the conditions in step $k3$ are not satisfied, then step $k2$ needs to be performed again with an updated $(\xi ,\lambda)$ obtained from the subroutine in step $k3$.  The conditions in step $k3$ are then checked again.  This will continue until the conditions are satisfied.  It should be noted however that the total number of times this occurs in an iteration is bounded.  In fact, if $N_0$ is the number of iterations taken by the algorithm before termination, then the total number of executions of steps $k2$ and $k3$ is bounded by $N_0 + \max \left\{ \frac{\log (\lambda_0/\underline{\lambda})}{\log \theta}, \frac{ \log \overline{\xi}}{\log 2}, 0 \right\}$. 

\vspace{10pt}

\noindent There is no particular reason for setting $A_0$ to be $12$ in the above algorithm.  We do this for the sake of convenience.  $A_0$ can be set to any positive number without affecting the results in this paper.

\vspace{10pt}

\noindent Note that in the above algorithm, for all $k \geq 1$, $y_k^{\rm{min}} \in {\rm{dom}}\ h$ and  $x_k, \tilde{x}_k \in \Omega$.  

\vspace{10pt}

\noindent We remark that for every $k \geq 1$, we have $y_k \in {\rm{dom}}\ h$ and $v_k \in \nabla f(y_k) + \partial h(y_k)$.  Hence, upon termination of the algorithm, we obtain an $\hat{\rho}$-approximate solution, $(\hat{y},\hat{v})$, of Problem (\ref{CNO}).

\vspace{10pt}

%\begin{proof}
%	The proof of the first inequality can be found in Lemma 7 of \cite{nesterov2012gradient}.
%	It is easy to see that the second inequality is a direct result of the first one.
%	Using the relation (\ref{Akak}), the first inequality of the lemma and the Cauchy-Schwartz inequality, we conclude that the last inequality of the lemma
%	holds as follows:
%	\[
%	\frac{\sum_{i=0}^{k-1}a_i}{\sum_{i=0}^{k-1}A_{i+1}}=\frac{\sum_{i=0}^{k-1}a_i}{\sum_{i=0}^{k-1}a_i^2}
%	\le \frac{\sum_{i=0}^{k-1}a_i}{\frac1k (\sum_{i=0}^{k-1}a_i)^2}
%	= \frac{k}{A_k}
%	\le \frac{4}{k}.
%	\]
%\end{proof}

\noindent Also, we remark that $\lambda_k$ in the algorithm can be viewed as an estimation of the reciprocal of $\overline{M}$, while $\xi_k$ is an estimation of $\underline{m}$.  Hence, when $f$ is convex, in which case $\underline{m} = 0$, $\xi_k = 0$, which also implies that $\tau_k = 0$ for all $k \geq 1$.  The algorithm then reduces to FISTA with constant stepsize \cite{beck2009fast} on Problem (\ref{CNO}) when we set $\lambda_k = 1/\overline{M}$ for all $k \geq 0$.

\vspace{10pt}

\noindent Note that $\{a_k\}$ and $\{A_k\}$ in (\ref{eq:all ak}) are related by
\begin{eqnarray}\label{Akak}
A_{k+1} = a_k^2.
\end{eqnarray}

\noindent Furthermore, we observe that by defining for $k \geq 1$
\begin{align}
& \tilde{\gamma}_k(u) := l_f(u; \tilde{x}_k) + h(u) + \frac{\tau_k}{2\lambda_k} \|u - \tilde{x}_k \|^2, \ u \in {\rm{dom}}\ h, \label{def:tildegamma} \\
& \gamma_k(u) := \tilde{\gamma}_k(y_k) + \frac{1}{\lambda_k} \langle \tilde{x}_k - y_k, u - y_k \rangle + \frac{\tau_k}{2 \lambda_k} \| u - y_k \|^2, \ u \in \Omega,  \label{def:gamma}
\end{align}
it is easy to check that $x_k$ given in (\ref{def:xk}) is the unique optimal solution to the following optimization problem.
\begin{eqnarray}\label{optimizationproblem}
\min_{u \in \Omega} a_{k-1}\gamma_k(u) + \frac{1}{2 \lambda_k} \| u - {x}_{k-1} \|^2.
\end{eqnarray}
In the definition of $\gamma_k$ in (\ref{def:gamma}), we note that aside from the quadratic term, $\gamma_k$ is the ``linearization" of $\tilde{\gamma}_k$ at $u = y_k$.

\section{Iteration Complexity Results for VAR-FISTA}\label{sec:complexity}

In this section, we derive iteration complexity results as stated in Theorem \ref{thm:main} to find an $\hat{\rho}$-approximate solution to Problem (\ref{CNO}) using VAR-FISTA.  %The lemmas and their proofs that lead the theorem are similar to results in \cite{Liang}.  For example, Lemma \ref{lem:basicinequality} is almost identical to parts of Lemma 2.2 and Lemma 2.3 in \cite{Liang}, and we omit its proof which can be deduced from proofs of these latter lemmas.

\vspace{10pt}

\noindent First, we define
\begin{eqnarray}\label{diameter}
D_h := \sup_{u_1,u_2\in {\rm{dom}}\ h} \| u_1 - u_2 \|.
\end{eqnarray}
Note that $D_h$ is finite since ${\rm{dom}}\ h$ is bounded.

\vspace{10pt}

\noindent We need the following results on $\{a_k\}$ and $\{A_k\}$ in deriving these iteration complexity results.  The proof of the lemma is given in the appendix.
 \begin{lemma}\label{estimates}
	For every $ k\ge 1 $, the sequences $ \{a_k\} $ and $ \{A_k\} $ given in \eqref{eq:all ak} satisfy
	\[
	\frac{k}{2} \leq a_{k-1} \leq 4k, \quad \sum_{i=1}^{k}A_{i}\ge\frac{k^3}{12}, \quad \frac{\sum_{i=1}^{k}a_{i-1}}{\sum_{i=1}^{k}A_{i}}\le \frac{4}{k}.
	\]
\end{lemma}

\noindent In the following lemma, we put together properties of $\tau_k, y_k^{\rm{min}}, \lambda_k, U_k, L_k$ and $\xi_k$ in VAR-FISTA.  These results are useful in our analysis later.
\begin{lemma}\label{observation}
The following statements hold for VAR-FISTA:
	\begin{itemize}
		\item[(a)]
		$\{\lambda_k\} $ is positive, non-increasing; $\{\xi_k\} $ and $\{ L_k \}$ are non-negative, non-decreasing;
		\item[(b)] for every $k \ge 1$,
		\begin{align*}
		& U_k \leq \overline{M}, \quad L_k\le \underline{m}, \quad  \tau_k = \frac{2 \xi_{k}\lambda_{k}}{a_{k-1}}, \quad U_k\lambda_{k} \le \gamma, \\
		& \xi_{k}\lam_{i-1} \ge L_k \lambda_{i} + \tau_i \geq L_i\lam_{i}+\tau_i\ge0,\ i = 1, \ldots, k, \\
		& y_{k}^{\min} =\argmin \left\{  \phi(\tilde{y})\ ;\ \tilde{y} \in \{ y_0,\ldots,y_{k} \}  \right\};
		\end{align*}
		\item[(c)] for every $k \ge 0$, $ \lam_k\ge \underline \lam :=\min\{\gamma/(\theta \overline{M}), \lam_0\} $, $ \xi_k\le  \max\{4\underline{m}, 1 \} $;  furthermore, if $f$ is convex, then $\xi_k = 0$ for every $k \geq 0$.  
%		\item[(d)]
%		the total number of times $\lam$ is updated according to \eqref{eq:lamupdate}  and the number of times $ \xi $ is updated according to \eqref{eq:xiupdate} throughout the whole method is bounded by $ {\cal O}\left( \log^+(\lam_0\bar M) \right)$ and ${\cal{O}} \left( \log^+(\bar{m}/\xi_0) \right)$ respectively.
	\end{itemize}
\end{lemma}

\begin{proof}
	(a)  The first statement follows from $\lambda_0 > 0$, the assumption that $\theta>1$ and the fact that the update procedure for $\lam$ in step $k3$ of VAR-FISTA
	either leaves $\lam$ unchanged
	or strictly decreases $\lam$ according to the update formula (\ref{eq:lamupdate}).
	That $\{ \xi_k \}$ is non-negative, non-decreasing is obvious in view of (\ref{eq:xiupdate0}) and (\ref{eq:xiupdate}) in step $k3$ of the algorithm, while $\{ L_k \}$ is non-negative, non-decreasing hold due to (\ref{def:L}) in step $k2$ of the algorithm.

\vspace{5pt}

\noindent (b) Since $\overline{M} > 0$ and $\underline{m} \geq 0$ satisfies (\ref{ineq:lipschitzconsequence}) and (\ref{lowercurvature}) respectively, it follows that every quantity $U$ (resp., $ L $) computed in step $k2$ of VAR-FISTA, and hence $U_k$ (resp., $ L_k $), is bounded above by $\overline{M}$ (resp., $\underline{m}$).
	The other conclusions follow immediately from (a) and the definitions of $ \tau_k $, $ y^{\min}_{k}$, $y_{k}$, $\lambda_{k}$, $U_k$, $ L_k $ and $ \xi_{k} $ in step $k3$ of VAR-FISTA. 

\vspace{5pt}

\noindent (c) For contradiction, assume that $ \lam_k<\underline \lam $ for some $ k\ge 0 $.
	Then, since $\lam_k< \lam_0$, $\lam_k$ has been obtained from a pair {\color{black}{$(\lam, U)$}} through the update formula (\ref{eq:lamupdate}) {\color{black}{and we also have $U > 0$}}.
	Since {\color{black}{$\overline{M} \ge U > 0$}} and $\lam_k < \gamma/(\theta M)$, it follows that {\color{black}{$\gamma/U \ge \gamma/\overline{M} > \lam_k$}}. Hence, it follows from (\ref{eq:lamupdate}) that $\lam_k= \lam/\theta$.
	On the other hand,  noting  that  step $k3$ in VAR-FISTA implies that $\lam$ is no longer reduced whenever $ \lam \le \gamma/\overline{M} $,
	we then conclude that $\lam > \gamma/\overline{M}$, and hence that $\lam_k = \lam/\theta > \gamma/(\theta \overline{M})$. Since the latter conclusion contradicts our initial assumption,  the first result in 
	statement (c) follows.  To show the second result in statement (c), for contradiction, assume that $\xi_k > \max \{ 4{\underline{m}}, 1 \}$ for some $k \geq 0$.  Since $\xi_k > 1$, we have $k \geq 1$, and we also have $\xi_k = 2 \xi$, where $\xi$ satisfies $\xi \lambda_{k-1} < L\lambda + \tau$ or $\xi \lambda_{i-1} < L \lambda_{i} + \tau_i$ for some $i = 1, \ldots, k-1$, according to (\ref{ineq:check}).   By $L \leq {\overline{m}}$ from (\ref{lowercurvature}) and (\ref{def:L}), definition of $\tau$ and $\tau_i$, $a_i \geq a_0 = 4$,  $\lambda \leq \lambda_{k-1}$, $\lambda_{i} \leq \lambda_{i-1}$ and $\xi \geq \xi_{i}, i = 1, \ldots, k-1$, we have $L \lambda + \tau \leq {\underline{m}} \lambda_{k-1} + (\lambda_{k-1} \xi)/2$ or $L \lambda_{i} + \tau_i \leq {\overline{m}} \lambda_{i-1} + (\lambda_{i-1} \xi)/2$.  Hence, $\xi \lambda_{i} < \lambda_{i}({\underline{m}} + \xi/2)$ for some $i = 1 ,\ldots, k-1$, which implies that $\xi < 2{\underline{m}}$.  Therefore, $\xi_k = 2 \xi < 4{\underline{m}}$, which contradicts our initial assumption. The second result in statement (c) then follows.  Furthermore, if $f$ is convex, then $\xi = \tau = \tau_i = L = 0$ and hence (\ref{ineq:check}) is always false, which implies that (\ref{eq:xiupdate0}) and (\ref{eq:xiupdate}) are never executed.  Therefore, $\xi_k = \xi_0 = 0$ for every $k \geq 1$.
\end{proof}

\vspace{10pt}

\noindent Lemma \ref{observation} is similar to Lemma 3.1 in \cite{Liang}.

\begin{remark}\label{rem:Remark1}
VAR-FISTA is able to ``detect" when $f$ is convex, in which case, $\xi_k$ is always equal to $0$ for all $k$, unlike when $f$ is nonconvex.  This leads to better iteration complexity for VAR-FISTA as shown in Theorem \ref{thm:main} below.  Although the algorithm performs differently in terms of update from $\xi_k$ to $\xi_{k+1}$ depending on whether $f$ is convex or nonconvex, we carry out the analysis to find the iteration complexity for VAR-FISTA in an unified manner.  We do this by defining $\bar{\xi}$ to be such that 
\begin{eqnarray}\label{def:barxi}
\bar{\xi} := \left \{ \begin{array}{ll}
\max\{4\underline{m}, 1 \}, & {\mbox{if}}\ {\underline{m}} > 0, \\
 0, & {\mbox{if}}\ {\underline{m}} = 0.
 \end{array} \right.
\end{eqnarray}
\noindent It is easy to see from the above definition of $\bar{\xi}$ that its value is zero only when $f$ is convex.  Observe also from (\ref{def:barxi}) and Lemma \ref{observation}(c) that $\forall\ k \geq 0$, $\xi_k \leq \bar{\xi}$.
\end{remark}

%\vspace{10pt}
%
%\noindent Let 
%\begin{eqnarray*}
%\bar{\xi} := \left \{ \begin{array}{ll}
% \max\{4\bar{m}, 1\}, & {\mbox{if}}\ {m} > 0, \\
% 0, & {\mbox{if}}\ {m} = 0.
% \end{array} \right.
%\end{eqnarray*}
%It follows from Lemma \ref{observation}(c) that $\xi_k \leq \bar{\xi}$ for every $k \geq 0$.

\noindent The following lemma is crucial for us to arrive at the iteration complexity results for VAR-FISTA in Theorem \ref{thm:main}.
\begin{lemma}\label{lem:iterations}
The total number of times, $n_0$, the value of $\xi_k$ changes as $k$ increases is of the order $\max \{ \log {\underline{m}}, 1 \}$.  
\end{lemma}
\begin{proof}
We observe that if $\xi \geq 2\underline{m}$, the inequalities in (\ref{ineq:check}) do not hold, which follows from Lemma \ref{observation}(a), $a_k \geq a_0 = 4$ and that $L$ in (\ref{def:L}) is always less than or equal to $\underline{m}$.  This, together with the update formula (\ref{eq:xiupdate}), leads to the result in the lemma.  
\end{proof}

%Let the constant that $\xi_{i+1}$ takes for $i \geq k_0$ be denoted by $\hat{\xi}$.  We observe that $\hat{\xi} \leq \bar{\xi}$, by Lemma \ref{observation}(c).

\vspace{10pt}
%
%\noindent When $\hat{\rho}=0$, let $K_0$ be the number of iterations executed before VAR-FISTA terminates. We have $K_0$ depends on $(\lambda_0, \theta, y_0)$ and problem data, and can be infinity, in which case, the algorithm does not terminate. 
\noindent The following lemma provides a bound on $\| x_k - x_0 \|$.

\begin{lemma}\label{lem:boundonxk}
We have for $k \geq 0$, $\|x_{k} - x_0 \| \leq Ck$, where 
\begin{eqnarray*}
C := 2(2 + \bar{\xi} \lambda_0)D_h.
\end{eqnarray*}
\end{lemma}
\begin{proof}
We have for $k \geq 1$, by (\ref{eq:all ak}), (\ref{def:xk}), (\ref{Akak}), Lemma \ref{observation}, (\ref{diameter}), the last sentence in Remark \ref{rem:Remark1} and Lemma \ref{estimates}, that
\begin{eqnarray*}
\| x_{k} - x_0 \| & \leq & \left\| \frac{(1+ \tau_k)A_{k}}{a_{k-1}(\tau_k a_{k-1} + 1)}y_{k} - \frac{A_{k-1}}{a_{k-1}(\tau_k a_{k-1} + 1)}y_{k-1} - x_0 \right\| \\
& = & \frac{1}{(2\xi_{k}\lambda_{k} + 1)a_{k-1}} \left\| (1+ \tau_k)A_{k}y_{k} - A_{k-1} y_{k-1} - (\tau_k a_{k-1} + 1)a_{k-1} x_0 \right\| \\
& \leq & \frac{1}{a_{k-1}}\left\| (1+ \tau_k)A_{k}y_{k} - A_{k-1} y_{k-1} - (\tau_k a_{k-1} + 1) a_{k-1} x_0 \right\| \\
& = &  \frac{1}{a_{k-1}} \left\| A_{k-1}(y_{k} - y_{k-1}) + (\tau_k a_{k-1} + 1)a_{k-1}(y_{k} - x_0) \right\| \\
& \leq & \frac{D_h}{a_{k-1}}( A_{k-1} + (\tau_k a_{k-1} + 1)a_{k-1}) \\
& = & D_h\left( \frac{A_{k-1}}{a_{k-1}} + 2\xi_{k}\lambda_{k} + 1 \right) \leq  D_h\left( \frac{A_{k-1}}{a_{k-1}} + 2{\bar{{\xi}}}\lambda_0 + 1 \right) \\
& \leq & D_h(a_{k-1} + 2{\bar{{\xi}}}\lambda_0) \leq  D_h(4k + 2 \overline{\xi} \lambda_0) \leq  2(2 + \overline{\xi} \lambda_0)D_h k.
%& \leq & \left(\frac{K D_h}{2\xi_0\underline{\lambda} + 1}\right)(k+1),
\end{eqnarray*}
%Now,  
%\begin{eqnarray}\label{akbound}
%a_{k-1} = \frac{1+ \sqrt{1 + 4 A_{k-1}}}{2} \leq 2 \sqrt{A_{k-1}}.
%\end{eqnarray}
%Hence,
%\begin{eqnarray*}
%A_{k} = A_{k-1} + a_{k-1} \leq A_{k-1} + 2 \sqrt{A_{k-1}} \leq (\sqrt{A_{k-1}} + 1)^2.
%\end{eqnarray*}
%Therefore,
%\begin{eqnarray}\label{Akbound}
%\sqrt{A_{k}} \leq \sqrt{A_{k-1}} + 1 \leq \sqrt{A_0} + k.
%\end{eqnarray}
%From (\ref{akbound}), using (\ref{Akbound}), we have
%\begin{eqnarray*}
%a_{k-1} \leq 2\sqrt{A_0} + 2(k-1) = 4\sqrt{3} + 2(k-1).
%\end{eqnarray*}
%It then follows from (\ref{xkbound}) that
%\begin{eqnarray*}
%\| x_{k} - x_0 \| \leq D_h(4 \sqrt{3} +2k +  2{\bar{{\xi}}}\lambda_0 - 1) \leq 2(5 + \bar{\xi} \lambda_0)D_h k.
%\end{eqnarray*}
The conclusion of the lemma then follows.
\end{proof}

\vspace{10pt}

%\noindent The following lemma is an adaptation of Lemma 2.2 and Lemma 2.3 in \cite{Liang}, and its proof, which is omitted, can be derived from the proofs of these lemmas in \cite{Liang}.

\noindent Below are two technical results that are needed in the analysis to arrive at Theorem \ref{thm:main}.  Proposition \ref{prop:technical} is used to prove Lemma \ref{lem:xi}.
\begin{proposition}\label{prop:technical}
For $u \in \ {\rm{dom}}\ h$, for every $k \geq 1$, we have
\begin{align}
& A_{k-1} \| y_{k-1} - \tilde{x}_k \|^2  \leq  2 \| u - x_{k-1} \|^2 + 2 D_h^2, \label{ineq:tech} \\
& a_{k-1} \| u - \tilde{x}_k \|^2  \leq  \frac{2}{a_{k-1}} \| u - x_{k-1} \|^2 + 2 a_{k-1} D_h^2. \label{ineq:tech1}
\end{align}
\end{proposition}
\begin{proof}
\noindent By the definition of $\tilde{x}_k$ in \eqref{eq:all ak},
relations \eqref{diameter} and (\ref{Akak}),
the fact that $A_{k} = A_{k-1} + a_{k-1} \ge A_{k-1}$ due to \eqref{eq:all ak}, the inequality $\|a+b\|^2 \le 2(\|a\|^2+\|b\|^2)$ for any $a,b \in \mathbb{R}^n$, we obtain for $u \in {\rm{dom}}\ h$,
%First note that 
\begin{eqnarray*}
A_{k-1} \| y_{k-1} - \tilde{x}_k \|^2 & = & \frac{A_{k-1} a_{k-1}^2}{A_{k}^2} \| x_{k-1} - y_{k-1} \|^2 =  \frac{A_{k-1}}{A_{k}} \| (x_{k-1} -u) - (y_{k-1}-u )\|^2 \\
& \leq  &  \frac{2A_{k-1}}{A_{k}}\left [ \| u - x_{k-1} \|^2 +  \| u - y_{k-1} \|^2 \right]  \leq  2 \|u - x_{k-1} \|^2 + 2D^2_h. % \label{equalityinequality}
\end{eqnarray*}	
Hence, (\ref{ineq:tech}) holds.  Arguing in a similar manner, (\ref{ineq:tech1}) holds as well. 
%\begin{align}
%&  A_{i-1} \| y_{i-1} - \tilde{x}_i \|^2 + a_{i-1} \| u - \tilde{x}_i\|^2 \nonumber \\
%& =  \frac{A_{i-1} a_{i-1}^2}{A_{i}^2} \| x_{i-1} - y_{i-1} \|^2 + a_{i-1} \left\| \frac{A_{i-1}}{A_{i}}(u - y_{i-1}) + \frac{a_{i-1}}{A_{i}}(u - x_{i-1}) \right\|^2 \nonumber \\
%& =  \frac{A_{i-1}}{A_{i}} \| (x_{i-1} -u) - (y_{i-1}-u )\|^2 + a_{i-1} \left\| \frac{A_{i-1}}{A_{i}}(u - y_{i-1}) + \frac{a_{i-1}}{A_{i}}(u - x_{i-1}) \right\|^2 \nonumber \\
%& \leq     \frac{2A_{i-1}}{A_{i}}\left [ \| u - x_{i-1} \|^2 +  \| u - y_{i-1} \|^2 \right]  + 2 a_{i-1} \left[ \frac{ A_{i-1}^2}{A_{i}^2} \| u - y_{i-1} \|^2 + \frac{a_{i-1}^2}{A_{i}^2} \| u - x_{i-1}\|^2  \right] \nonumber \\
%& \leq   \frac{2A_{i-1}}{A_{i}}\ \left [ \| u - x_{i-1} \|^2 +  \| u - y_{i-1} \|^2 \right]  + 2 a_{i-1} \| u - y_{i-1} \|^2  + \frac{2 a_{i-1}}{A_{i}} \|u - x_i\|^2 \nonumber \\
%& \leq  2 \|u - x_{i-1} \|^2 + 2(1 + a_{i-1}) D^2_h, \label{ineq:tech2} % \label{equalityinequality}
%\end{align}	
\end{proof}

\vspace{10pt}

\noindent The following technical result allows us to arrive at Lemma \ref{lem:basicinequality}, which through Lemma \ref{lem:xi}, then leads to Theorem \ref{thm:main}, the main result of this section.  This proposition is also needed in the proof of Lemma \ref{lem:xi}.  The proof of this proposition and that of Lemma \ref{lem:basicinequality} are similar to that of Lemma 2.2 and Lemma 2.3 in \cite{Liang} respectively, and are provided in the appendix of this note for the sake of completeness.

\begin{proposition}\label{prop:technical1}
%The following statements hold:
%\begin{itemize}
%\item[(a)] 
$\tilde{\gamma}_k$ defined in (\ref{def:tildegamma}) and $\gamma_k$ defined in (\ref{def:gamma}) are $(\tau_k/\lambda_k)$-strongly convex functions, $\gamma_k(u) \leq \tilde{\gamma}_k(u) \ \forall\ u \in {\rm{dom}}\ h$, $\tilde{\gamma}_k(y_k) = \gamma_k(y_k)$,
\begin{eqnarray}\label{eq:minimizationequality}
\min_u \left\{ \tilde{\gamma}_k(u) + \frac{1}{2\lambda_k} \| u - \tilde{x}_k \|^2 \right \} = \min_u \left\{  \gamma_k(u) + \frac{1}{2\lambda_k} \| u - \tilde{x}_k \|^2 \right\},
\end{eqnarray}
and these minimization problems have $y_k$ as their unique optimal solution;
%\item[(b)] for $k \geq 1$, for every $u \in {\rm{dom}}\ h$, 
%\begin{eqnarray*}
%\tilde{\gamma}_k(u) - \phi(u) \leq \frac{1}{2} \left( \underline{m} + \frac{\tau_k}{\lambda_k} \right) \|u - \tilde{x}_k \|^2.
%\end{eqnarray*}
%\end{itemize}
\end{proposition}

\begin{lemma}\label{lem:basicinequality}
%Define for $k \geq 0$ and $u \in {\rm{dom}}\ h$,
%\begin{eqnarray}\label{def:tildegamma}
%\tilde \gamma_k(u)  :=  \ell_f(u;\tilde{x}_k) + h(u) + \frac{\tau_k}{2 \lambda_{k+1}} \| u - \tilde{x}_k \|^2,
%\end{eqnarray}
%and $u \in \Omega$,
%\begin{eqnarray}\label{def:gamma}
%\gamma_k(u) := \tilde\gamma_k(y_{k+1}) +  \frac{1}{\lam_{k+1}}\langle \tilde{x}_k - y_{k+1}, u - y_{k+1} \rangle + \frac{\tau_k}{2 \lambda_{k+1}} \| u - y_{k+1} \|^2. 
%\end{eqnarray}
We have for $k \geq 1$, for every $u \in \Omega$,
	\begin{align}
%	\begin{aligned}
		& \lam_{k} A_{k} \phi(y_{k}) + \frac{ \tau_k a_{k-1} + 1}{2} \| u - x_{k} \|^2 + \frac{(1 - \gamma) A_{k}}{2} \| y_{k} - \tilde{x}_{k} \|^2  \nonumber\\
		& \leq  \lam_{k} A_{k-1} \gamma_{k}(y_{k-1}) + \lam_{k}  a_{k-1} \gamma_{k}(u) + \frac{1}{2} \|u - x_{k-1} \|^2,  \label{ineq:recursive}
%	\end{aligned}
	\end{align}
\end{lemma}

\noindent The inequality (\ref{ineq:recursive}) in the above lemma is the basic inequality fundamental in proving the iteration complexity results for VAR-FISTA, and is the key result needed to show that the following lemma holds.

\begin{lemma}\label{lem:xi} 
For every $ k\ge 1$,
\begin{align}
	& \frac{1-\gamma}{2}\left( \sum_{i=1}^{k}A_{i}\| y_{i} - \tilde{x}_i \|^2 \right) \le 
	 \lam_{0} A_{0}(\phi(y_{0}) - \phi(y_{k}^{\rm{min}}))+ \left(\frac{1}{2} + 2\overline{\xi}\lambda_0n_0\right) D_h^2  \nonumber \\
	 & +  2\bar{\xi} \lam_0 C^2\left(\sum_{i=1}^k \frac{i^2}{a_{i-1}} + n_0 k^2 \right) + \overline{\xi}\lambda_0 D_h^2 \sum_{i=1}^{k} (3 + a_{i-1}). \label{ineq:key4}
\end{align}
\end{lemma}
\begin{proof}
For $k \geq 1$, let $i_j \leq k$, $j \geq 1$, be such that $\xi_{i} = \xi_{i_{j}}$ for $i = i_{j-1} +1, \ldots, i_{j}$, where $i_0 = 0$ and $i_{n_1} = k$.  Note that $n_1 \leq  n_0$, by Lemma \ref{lem:iterations}.  From \eqref{ineq:recursive} in Lemma \ref{lem:basicinequality}, where we let $u = y_{k}^{\rm{min}}$, for $i_{j-1} + 1 \leq i \leq i_j$, we have
\begin{align}
	&  \frac{1 - \gamma}{2}A_{i} \| y_{i} - \tilde{x}_i \|^2 \nonumber \\
	%	\le    \frac{(1 -\lam_{i+1} M_i) A_{i+1}}{2} \| y_{i+1} - \tilde{x}_i \|^2 \\
	& -  \left( \lam_{i-1} A_{i-1}(\phi(y_{i-1}) - \phi( y^{\min}_{k})) + \frac{1}{2} \|  y^{\min}_{k} - x_{i-1} \|^2 \right)
	+\left( \lam_{i}A_{i} (\phi(y_{i}) - \phi( y^{\min}_{k})) + \frac{1}{2} \|  y^{\min}_{k} - x_{i} \|^2 \right) \nonumber \\
	&   \le  \lam_{i} A_{i-1}(\gamma_i(y_{i-1}) - \phi(y_{i-1})) + \lam_{i} a_{i-1}(\gamma_i( y^{\min}_{k}) - \phi( y^{\min}_{k})) - \frac{ \tau_i a_{i-1}}{2} \|  y^{\min}_{k} - x_{i} \|^2 \nonumber \\
	& + \left( \lam_{i} - \lam_{i-1}\right)  A_{i-1}\left( \phi(y_{i-1}) - \phi(y^{\min}_{k})\right). \label{ineq:consequence} 
\end{align}
Observe that by Proposition \ref{prop:technical1}, the definition of $\tilde{\gamma}_i$ in (\ref{def:tildegamma}), and $L_i$ in view of (\ref{def:L}) that for $i = i_{j-1} + 1, \ldots, i_{j}$, %when $u = y_{i-1}$, we have
\begin{eqnarray}
	\gamma_i(y_{i-1})-\phi(y_{i-1}) & \le & \tilde \gamma_i(y_{i-1})-\phi(y_{i-1})
	=\ell_f(y_{i-1};\tx_i)-f(y_{i-1})+\frac{\tau_i}{2\lam_{i}}\|y_{i-1}-\tx_i\|^2 \nonumber \\
&	\le & \left( \frac{L_{i}}{2}+\frac{\tau_i}{2\lam_{i}}\right) \|y_{i-1}-\tx_i\|^2,  \label{ineq:diff}
\end{eqnarray}
and %when $u= y_{k}^{\rm{min}}$, we have
\begin{eqnarray}
\gamma_i(y_{k}^{\rm{min}})-\phi(y_{k}^{\rm{min}}) & \le & \tilde \gamma_i(y_{k}^{\rm{min}})-\phi(y_{k}^{\rm{min}})
	=\ell_f(y_{k}^{\rm{min}};\tx_i)-f(y_k^{\rm{min}})+\frac{\tau_i}{2\lam_{i}}\|y_{k}^{\rm{min}}-\tx_i\|^2 \nonumber \\
&	\le  & \left( \frac{L_{k}}{2}+\frac{\tau_i}{2\lam_{i}}\right) \|y_{k}^{\rm{min}}-\tx_i\|^2.  \label{ineq:diff1}
\end{eqnarray}
%	Following the proof of Lemma 2.3 and the third inequality in Lemma \ref{observation}(b), it is easy to see that for $k \geq 0$, $u \in \Omega$,
%	\begin{align}
%%	\begin{aligned}
%		& \lam_{k+1} A_{k+1} \phi(y_{k+1}) + \frac{ \tau_k a_k + 1}{2} \| u - x_{k+1} \|^2 + \frac{0.1 A_{k+1}}{2} \| y_{k+1} - \tilde{x}_k \|^2  \nonumber\\
%		& \leq  \lam_{k+1} A_k \gamma_k(y_k) + \lam_{k+1}  a_k \gamma_k(u) + \frac{1}{2} \|u - x_k \|^2,  \label{ineq:recursive}
%%	\end{aligned}
%	\end{align}
%	where $\tilde{\gamma}_k(u)$ and ${\gamma}_k(u)$ are given by (16) and (17) with $\lambda = \lambda_{k+1}$.	From Lemma 2.2(a), we have $ \gamma_i(u)\le \tilde \gamma_i(u) $ for every $ u\in \dom h $. 
%It follows from Lemma \ref{lem:basicinequality}, the definition of $ \tilde \gamma_i $, and $ L_k $ in view of \eqref{def:L} that, for every $k\ge 0, 0 \leq i \leq k$, when $u = y_i$ or $y^{\min}_{k+1}$, we have

%	&  \le \frac12\left( L_k \lam_{i+1} + \tau_i \right)  \left( A_i\| y_i - \tx_i \|^2 + a_i\| y^{\min}_{k+1} - \tx_i \|^2 \right)   - \frac{ \tau_i a_i}{2}  \|  y^{\min}_{k+1} - x_{i+1} \|^2  \nonumber  \\
%	&  \le \left( L_k \lam_{i+1} + \tau_i \right)  \left( \| y^{\min}_{k+1} - x_i \|^2 + (1 + a_i) D_h^2 \right)   - \frac{ \tau_i a_i}{2}  \|  y^{\min}_{k+1} - x_{i+1} \|^2 . \label{ineq:key1}

\noindent From (\ref{ineq:consequence}), for $i_{j-1} + 1 \leq i \leq i_j$, using \eqref{ineq:diff}, (\ref{ineq:diff1}), $\{\lambda_i \}$ is non-increasing in view of Lemma \ref{observation}(a), $ \phi(y^{\min}_{k})\le\phi(y_{i-1})$,  Proposition \ref{prop:technical} where $u = y_k^{\rm{min}}$, $0 \le L_{i}\lam_{i}+\tau_i\le \xi_{i_j} \lambda_{i-1}$ and $0 \leq L_k \lambda_i + \tau_i \leq \xi_k \lambda_{i-1}$ in view of Lemma \ref{observation}(b), $\tau_i = 2 \xi_i \lambda_i/a_{i-1}$, $\xi_{i} = \xi_{i_j}$, $\lambda_{j-1} \leq \lambda_0$ and the last statement in Remark \ref{rem:Remark1}, we conclude that 
\begin{align}
	&  \frac{1 - \gamma}{2}A_{i} \| y_{i} - \tilde{x}_i \|^2 \nonumber \\
	%	\le    \frac{(1 -\lam_{i+1} M_i) A_{i+1}}{2} \| y_{i+1} - \tilde{x}_i \|^2 \\
	& -  \left( \lam_{i-1} A_{i-1}(\phi(y_{i-1}) - \phi( y^{\min}_{k})) + \frac{1}{2} \|  y^{\min}_{k} - x_{i-1} \|^2 \right)
	+\left( \lam_{i}A_{i} (\phi(y_{i}) - \phi( y^{\min}_{k})) + \frac{1}{2} \|  y^{\min}_{k} - x_{i} \|^2 \right) \nonumber \\
	&   \le  \frac{1}{2}(L_{i} \lambda_i + \tau_i)A_{i-1} \| y_{i-1} - \tilde{x}_i \|^2 + \frac{1}{2} (L_k \lambda_i + \tau_i) a_{i-1} \| y_{k}^{\rm{min}}  - \tilde{x}_i \|^2 - \frac{ \tau_i a_{i-1}}{2} \|  y^{\min}_{k} - x_{i} \|^2 \nonumber \\
	& \le  (L_{i} \lambda_i + \tau_i)( \| y_{k}^{\rm{min}} - {x}_{i-1} \|^2 + D_h^2) + (L_k \lambda_i + \tau_i)\left(\frac{1}{a_{i-1}} \|y_k^{\rm{min}} - x_{i-1} \|^2 + a_{i-1} D_h^2 \right) \nonumber \\
	& - \frac{ \tau_i a_{i-1}}{2} \|  y^{\min}_{k} - x_{i} \|^2 \nonumber \\
	& \le \xi_{i_j}\lambda_{i-1}( \| y_{k}^{\rm{min}} - {x}_{i-1} \|^2 + D_h^2) + \xi_k \lambda_{i-1}\left(\frac{1}{a_{i-1}} \|y_k^{\rm{min}} - x_{i-1} \|^2 + a_{i-1} D_h^2 \right) - \xi_i \lambda_i  \|  y^{\min}_{k} - x_{i} \|^2 \nonumber \\
	& \le \xi_{i_j}(\lambda_{i-1}\| y_{k}^{\rm{min}} - {x}_{i-1} \|^2 -  \lambda_i  \|  y^{\min}_{k} - x_{i} \|^2) + \overline{\xi}\lambda_0\left( \frac{1}{a_{i-1}} \| y_k^{\rm{min}} - x_{i-1} \|^2 +  (1 + a_{i-1})D_h^2 \right). \label{ineq:consequence1} 
\end{align}
%
%	For $0 \leq i \leq k$, from Lemma \ref{observation}(b), $\lambda_i \leq \lambda_0$ in Lemma \ref{observation}(a), and $\xi_{k+1}\le \bar \xi$ in (\ref{def:barxi}), we obtain
%	\begin{align}
%	& \left( L_k \lam_{i+1} + \tau_i \right)  \left( \| y^{\min}_{k+1} - x_i \|^2 + (1 + a_i) D_h^2 \right)   - \frac{ \tau_i a_i}{2}  \|  y^{\min}_{k+1} - x_{i+1} \|^2  \nonumber  \\
%	&\le  \xi_{k+1}\lam_i  \left( \| y^{\min}_{k+1} - x_i \|^2 + (1 + a_i) D_h^2 \right)   - \xi_{i+1}\lam_{i+1}  \|  y^{\min}_{k+1} - x_{i+1} \|^2 \nonumber \\
%	&\le \left[ \xi_{k+1}\lam_i \| y^{\min}_{k+1} - x_i \|^2 - \xi_{i+1}\lam_{i+1}  \|  y^{\min}_{k+1} - x_{i+1} \|^2 \right] + \bar{\xi} \lambda_0 (1 + a_i) D_h^2. \label{ineq:key2} 
%%	&\le \beta \left[ \lam_i \| y^{\min}_{k+1} - x_i \|^2 - \lam_{i+1}  \|  %y^{\min}_{k+1} - x_{i+1} \|^2 \right] + \bar{\xi} \lambda_0 (1 + a_i) D_h^2
%%	. 
%	\end{align}
\noindent Summing the inequality in (\ref{ineq:consequence1}) from $i = 1$ to $k$, we obtain
	\begin{align}
	& \frac{1-\gamma}{2}\left( \sum_{i=1}^{k}A_{i}\| y_{i} - \tilde{x}_i \|^2 \right) \le 
	 \lam_{0} A_{0}(\phi(y_{0}) - \phi(y_{k}^{\min}))+ \frac{1}{2} \|y_{k}^{\min} - x_{0} \|^2 \nonumber \\
	 &  +\bar \xi \lam_0 \sum_{i=1}^k \left( \frac{1}{a_{i-1}} \| y_k^{\rm{min}} - x_{i-1} \|^2 +  (1 + a_{i-1})D_h^2 \right)  + \sum_{j=1}^{n_1} \xi_{i_j}\lam_{i_{j}-1} \| y^{\min}_{k} - x_{i_{j}-1} \|^2 . \label{ineq:key3}
	\end{align}
%For $i \geq k_0$, by Lemma \ref{lem:iterations}, we have $\xi_{i+1} = \xi_{k_0+1}$.  Therefore, for $k \geq k_0$, we have
%\begin{align*}	
%	& \sum_{i=0}^k \left[ \xi_{k+1}\lam_i \| y^{\min}_{k+1} - x_i \|^2 - \xi_{i+1}\lam_{i+1}  \|  y^{\min}_{k+1} - x_{i+1} \|^2 \right] \\
%	& = \sum_{i=0}^{k_0-1} \left[ \xi_{k+1}\lam_i \| y^{\min}_{k+1} - x_i \|^2 - \xi_{i+1}\lam_{i+1}  \|  y^{\min}_{k+1} - x_{i+1} \|^2 \right] \\
%	& + \sum_{i=k_0}^{k} \left[ \xi_{k+1}\lam_i \| y^{\min}_{k+1} - x_i \|^2 - \xi_{i+1}\lam_{i+1}  \|  y^{\min}_{k+1} - x_{i+1} \|^2 \right] \\
%	& \leq \sum_{i=0}^{k_0-1} \left[ {\xi}_{k_0+1}\lam_i \| y^{\min}_{k+1} - x_i \|^2 - \xi_{i+1}\lam_{i+1}  \|  y^{\min}_{k+1} - x_{i+1} \|^2 \right] + \xi_{k_0+1} \lam_{k_0} \| y^{\min}_{k+1} - x_{k_0} \|^2 \\
%	& \leq 2\bar{\xi}\lambda_0 \sum_{i=0}^{k_0} \| y_{k+1}^{\min} - x_i \|^2 \leq 4 \bar{\xi} \lambda_0 \left( (k_0+1) \|y_{k+1}^{\min} - x_0\|^2 + \sum_{i=0}^{k_0} \|x_i - x_0\|^2 \right) \\
%	& \leq 4\bar{\xi}\lambda_0\left( (k_0+1)D_h^2 + C^2 \sum_{i=0}^{k_0} i^2 \right) \leq 4\bar{\xi}\lambda_0 C^2 \left(k_0 + 1 + \sum_{i=0}^{k_0}i^2 \right) \\
%	& = 4\bar{\xi}\lambda_0 C^2\left(k_0 + 1 + \frac{1}{6}k_0 (k_0 + 1)(2k_0 + 1) \right) \leq 6\bar{\xi}\lambda_0 C^2(k_0+1)^3.
%\end{align*}
Now, for $0 \leq i \leq k$, by Lemma \ref{lem:boundonxk},
\begin{eqnarray*}
\| y_k^{\rm{min}} - x_{i} \|^2 \leq 2(\| y_k^{\rm{min}} - x_0 \|^2 + \| x_0 - x_i \|^2) \leq 2 \| y_k^{\rm{min}} - x_0 \|^2 + 2C^2i^2.
\end{eqnarray*}
Therefore, by the above and that $\xi_{i_j} \lambda_{i_{j-1}} \leq \overline{\xi} \lambda_0$, we have from (\ref{ineq:key3}) 
\begin{align*}
	& \frac{1-\gamma}{2}\left( \sum_{i=1}^{k}A_{i}\| y_{i} - \tilde{x}_i \|^2 \right) \le 
	 \lam_{0} A_{0}(\phi(y_{0}) - \phi(y_{k}^{\rm{min}}))+ \left(\frac{1}{2} + 2\overline{\xi}\lambda_0 \left(n_1 + \sum_{i=1}^{k}\frac{1}{a_{i-1}}\right) \right) \|y_{k}^{\rm{min}} - x_{0} \|^2 \\
	 &  + 2\bar{\xi} \lam_0 C^2\left(\sum_{i=1}^k \frac{i^2}{a_{i-1}} + n_1 k^2 \right) + \overline{\xi}\lambda_0 D_h^2 \sum_{i=1}^{k} (1 + a_{i-1}).
\end{align*}
The conclusion of the lemma then follows by noting that $n_1 \leq n_0$, $a_{i-1} \geq 1$ and the definition of $D_h$ in (\ref{diameter}).
\end{proof}

\vspace{10pt}

\noindent We are now ready to state the iteration complexity results of VAR-FISTA to solve Problem (\ref{CNO}).

\begin{theorem}\label{thm:main}
VAR-FISTA terminates to obtain an $\hat{\rho}$-approximate solution $(\hat{y},\hat{v})$ to Problem (\ref{CNO}) in at most
\begin{eqnarray}\label{complexity}
\left(\frac{3C_1 L_1}{\hat{\rho}^2}\right)^{1/3} + \left(\frac{3C_1\bar{\xi}\lambda_0(2C^2 + 3 D_h^2)}{\hat{\rho}^2}\right)^{1/2} + \frac{C_1 \bar{\xi}\lambda_0(6C^2(1 + n_0) + D_h^2)}{\hat{\rho}^2} + 1
\end{eqnarray}
iterations, where
\begin{align*}
& C_1 =  \left(\frac{8}{1 - \gamma}\right) \left(\frac{1}{\underline{\lambda}} + \frac{1}{2} \bar{\xi} + \overline{M}  \right)^2, \\
& L_1 = \lam_0A_0(\phi(y_0)-\phi(y^\ast))+ \left( \frac{1}{2} + 2 \overline{\xi} \lambda_0 n_0 \right)D_h^2,
\end{align*}
and recall that $C = 2(2 + \bar{\xi} \lambda_0)D_h$ and $n_0 = \mathcal{O}(\max \{ \log \underline{m}, 1 \})$.  Furthermore, if $f$ in Problem (\ref{CNO}) is convex, then the iteration complexity of VAR-FISTA to solve the problem is improved, and it finds an $\hat{\rho}$-approximate solution $(\hat{y},\hat{v})$ to Problem (\ref{CNO}) in at most 
\begin{eqnarray}\label{complexity2}
\left(\frac{3C_2 L_2}{\hat{\rho}^2}\right)^{1/3}  + 1 
\end{eqnarray}
iterations, where
\begin{align*}
& C_2 = \left(\frac{8}{1 - \gamma} \right) \left( \frac{1}{\underline{\lambda}} + \overline{M} \right)^2, \\
& L_2 = \lambda_0A_0(\phi(y_0) - \phi(y^\ast)) + \frac{1}{2} D_h^2.
\end{align*}
\end{theorem}
\begin{proof}
	Using the facts that $ \{a_k\} $ is increasing, $a_0 = 4$, Lemma \ref{observation}(b), (c), and the last statement in Remark \ref{rem:Remark1}, we have for $k \geq 1$,
	\[
	\frac{1 + \tau_k}{\lambda_{k}} = \frac{1}{\lambda_{k}} + \frac{2 \xi_{k}}{a_{k-1}} \leq \frac{1}{\underline{\lambda}} + \frac{2 \bar{\xi}}{a_0} =  \frac{1}{\underline{\lambda}} + \frac{1}{2}\bar{\xi},
	\]
	%where $ \xi $ is defined in Lemma \ref{observation} (b), 
	and hence, together with (\ref{lipschtiz}), by \eqref{def:vk+1}, we obtain
	\[
	\min_{1 \leq i \leq k}\|v_i\| \le\min_{1 \leq i \leq k} \left( \frac{1 + \tau_i}{\lambda_{i}} + \overline{M} \right) \|y_{i}-\tx_i\| 
	\le \tilde C\min_{1 \leq i \leq k}\|y_{i}-\tx_i\|,
	\]
	where
	\[
	\tilde C= \frac{1}{\underline \lam}+\frac12\bar \xi+\overline{M}.
	\]
	Using the above inequality, (\ref{ineq:key4}) in Lemma \ref{lem:xi}, definition of $y^\ast$ and the first inequality in Lemma \ref{estimates}, we obtain for $k \geq 1$,
	\begin{align*}
	& \left( \frac{1 - \gamma}{2}\sum_{i=1}^{k} A_{i}\right) \min_{1 \leq i \leq k} \| v_i \|^2 \\ 
	&\le \tilde C^2 \left[
	\lam_{0} A_{0}(\phi(y_{0}) - \phi(y_{k}^{\rm{min}}))+ \left(\frac{1}{2} + 2\overline{\xi}\lambda_0n_0\right) D_h^2 +  2\bar{\xi} \lam_0 C^2\left(\sum_{i=1}^k \frac{i^2}{a_{i-1}} + n_0 k^2 \right) \right. \\
	& \left.  + \overline{\xi}\lambda_0 D_h^2 \sum_{i=1}^{k} (3 + a_{i-1}) \right] \\
	&\leq \tilde C^2 \left[\lam_0A_0(\phi(y_0)-\phi(y^\ast)) + \left(\frac{1}{2} + 2\overline{\xi}\lambda_0n_0\right) D_h^2 +  2\bar{\xi} \lam_0 C^2\left(\sum_{i=1}^k 2i + n_0 k^2 \right) \right. \\
	&  \left.  + \overline{\xi}\lambda_0 D_h^2 \left( 3k + \sum_{i=1}^{k} a_{i-1} \right) \right]. 
	\end{align*}
	Hence the complexity result \eqref{complexity} follows from the above inequality, the third and fourth inequality in Lemma \ref{estimates}.  The result (\ref{complexity2}) follows from (\ref{complexity}) and  $\bar{\xi}=0$ by (\ref{def:barxi}) where $\underline{m} = 0$ since $f$ is convex.
%	Consequently, \eqref{resolvent} follows from \eqref{cmplx} and Lemma \ref{observation2}(b).
\end{proof}

%\vspace{10pt}
%
%\noindent When $f$ is convex, by the above theorem, we have an improved iteration bound compared with when $f$ is nonconvex to find an $\hat{\rho}$-approximate solution to (\ref{CNO}).  %To our knowledge, the currently best known iteration-complexity for solving (\ref{CNO}) using accelerated gradient methods when $f$ is convex can be found for example in \cite{beck2009fast,nesterov2012gradient,tseng2008accmet}, and when $f$ is nonconvex can be found in \cite{nonconv_lan16}.

%\section{Numerical Study}\label{sec:numerical}

\section{Conclusion}\label{sec:conclusion}

In this note, we propose a first order algorithm, VAR-FISTA, to solve composite optimization problems, and establish iteration complexity result for the convex and nonconvex case in Theorem \ref{thm:main} that are best known in the literature so far.  We remark that even though the iteration complexity for the convex case is better than that for the nonconvex case as shown in Theorem \ref{thm:main}, implementation\footnote{We do not provide numerical results that we obtained in this note.} of the algorithm shows that the number of iterations to obtain an $\hat{\rho}$-approximate solution for instances of the quadratic programming problem as found in \cite{Liang} is worse when the instance is convex than when it is nonconvex.  This phenomenon appears to occur as well in \cite{Liang} for other first order methods tested in the paper. We do not have a reasonable explanation for this unusual phenomenon.

\bibliographystyle{plain}
%\bibliography{RADMM_ref2}
\bibliography{Proxacc_ref}
%}

%\begin{thebibliography}{9}
%%\bibitem{Bazaraa}
%%M. S. Bazaraa, H. D. Sherali \& C. M. Shetty, \emph{Nonlinear Programming: Theory and Algorithms,} $2^{nd}$ Edition, John Wiley \& Sons, Inc., %1993.
%\bibitem{Liang}
%J. Liang, R. D. C. Monteiro \& C.-K. Sim, \emph{A modified FISTA on nonconvex problems,} Preprint, 2019.
%%\bibitem{Rockafellar}
%%R. T. Rockafellar, \emph{Convex Analysis,} Princeton Landmarks in Mathematics, Princeton University Press, 1970.
%\end{thebibliography}

\appendix
\section{Appendix}\label{Appendix A}

\noindent {\bf{Proof of Lemma \ref{estimates}}}:  For $k \geq 1$, observe that
\begin{eqnarray*}
\frac{1}{2} + \sqrt{A_{k-1}} \leq a_{k-1} = \frac{1 + \sqrt{1 + 4A_{k-1}}}{2} \leq 2 \sqrt{A_{k-1}}.
\end{eqnarray*}
It follows that
\begin{align*}
& \left( \sqrt{A_{k-1}} + \frac{1}{2} \right)^2 \leq A_{k-1} + \sqrt{A_{k-1}} + \frac{1}{2} \leq A_k = A_{k-1} + a_{k-1} \\
& \leq A_{k-1} + 2 \sqrt{A_{k-1}} \leq (\sqrt{A_{k-1}} + 1)^2.
\end{align*}
Hence,
\begin{eqnarray*}
\sqrt{A_0} + \frac{k}{2} \leq \sqrt{A_{k-1}} + \frac{1}{2} \leq \sqrt{A_k} \leq \sqrt{A_{k-1}} + 1 \leq \sqrt{A_0} + k.
\end{eqnarray*}
Since $A_k = a_{k-1}^2$ and $A_0 = 12$, we conclude from the above that
\begin{eqnarray}\label{ineq:ak}
\frac{k}{2} \leq a_{k-1} \leq 4k.
\end{eqnarray}
Now, by $A_i = a_{i-1}^2$ and (\ref{ineq:ak}), we have
\begin{eqnarray*}
\sum_{i=1}^k A_i = \sum_{i=1}^{k} a_{i-1}^2 \geq \frac{1}{4} \sum_{i=1}^{k} i^2 = \frac{1}{24}k(k+1)(2k+1) \geq \frac{k^3}{12}.
\end{eqnarray*}
From $A_i = a_{i-1}^2$, $A_i = A_{i-1} + a_{i-1}$ and (\ref{ineq:ak}), we have
\begin{eqnarray*}
\frac{\sum_{i=1}^{k}a_{i-1}}{\sum_{i=1}^k A_i} = \frac{\sum_{i=1}^k a_{i-1}}{\sum_{i=1}^k a_{i-1}^2} \leq \frac{k \sum_{i=1}^k a_{i-1}}{\left( \sum_{i=1}^{k} a_{i-1} \right)^2} = \frac{k}{\sum_{i=1}^{k} a_{i-1}} = \frac{k}{A_k - A_0} \leq \frac{k}{a_{k-1}^2} \leq \frac{4}{k}.
\end{eqnarray*}
\hfill\hbox{\vrule width1.0ex height1.0ex}

\vspace{10pt}

\noindent {\bf{Proof of Proposition \ref{prop:technical1}}}: It is clear from the definition of $\tilde{\gamma}_k$ and $\gamma_k$ that they are $(\tau_k/\lambda_k)$-strongly convex. By (\ref{def:yk}), the way $y_k$ is defined in step $k3$ of VAR-FISTA and the definition of $\tilde{\gamma}_k$ in (\ref{def:tildegamma}), we see that $y_k$ is the optimal solution to the first minimization problem in (\ref{eq:minimizationequality}).  Since the objective function of this minimization problem is $((1 + \tau_k)/\lambda_k)$-strongly convex, it follows that $\forall\ u \in \Re^n$,
\begin{eqnarray}\label{ineq:stronglyconvex}
\tilde{\gamma}_k(y_k) + \frac{1}{2\lambda_k} \|y_k - \tilde{x}_k\|^2 + \frac{1 + \tau_k}{2\lambda_k} \|y_k - u\|^2 \leq \tilde{\gamma}_k(u) + \frac{1}{2\lambda_k} \|u - \tilde{x}_k \|^2.
\end{eqnarray}
On the other hand, the definition of $\gamma_k$ in (\ref{def:gamma}) and the relation
\begin{eqnarray*}
\| y_k - \tilde{x}_k \|^2 + \| y_k - u\|^2 = 2 \langle \tilde{x}_k - y_k, u - y_k \rangle + \| u - \tilde{x}_k \|^2
\end{eqnarray*}
imply that
\begin{eqnarray}\label{eq:tildegammagamma}
\tilde{\gamma}_k(y_k) + \frac{1}{2\lambda_k} \|y_k - \tilde{x}_k\|^2 + \frac{1 + \tau_k}{2\lambda_k} \|y_k - u\|^2 = \gamma_k(u) + \frac{1}{2\lambda_k}\| u - \tilde{x}_k \|^2.
\end{eqnarray}
Hence, comparing (\ref{ineq:stronglyconvex}) with (\ref{eq:tildegammagamma}), we have $\gamma_k(u) \leq \tilde{\gamma}_k(u) \ \forall \ u \in {\rm{dom}}\ h$, and from (\ref{eq:tildegammagamma}), we have $\tilde{\gamma}_k(y_k) = \gamma_k(y_k)$.  Furthermore, $\gamma_k(y_k) = \tilde{\gamma}_k(y_k)$, (\ref{eq:tildegammagamma}) and $\gamma_k \leq \tilde{\gamma}_k$ imply that
\begin{eqnarray*}
\gamma_k(y_k) + \frac{1}{2\lambda_k}\|y_k - \tilde{x}_k\|^2 & = & \tilde{\gamma}_k(y_k) + \frac{1}{2\lambda_k} \|y_k - \tilde{x}_k\|^2 \\
& \leq & \tilde{\gamma}_k(y_k) + \frac{1}{2\lambda_k} \|y_k - \tilde{x}_k\|^2 + \frac{1 + \tau_k}{2\lambda_k} \| y_k - u\|^2 \\
& = & \gamma_k(u) + \frac{1}{2\lambda_k} \|u - \tilde{x}_k \|^2 \leq \tilde{\gamma}_k(u) + \frac{1}{2\lambda_k} \| u - \tilde{x}_k \|^2
\end{eqnarray*}
for all $u \in \Re^n$, and hence the remaining conclusions of (a) follow.
%
%\vspace{5pt}
%
%\noindent (b) This statement follows from the definition of $\tilde{\gamma}_k$ in (\ref{def:tildegamma}), (\ref{lowercurvature}) and the definition of $\underline{m}$. 
%
\hfill\hbox{\vrule width1.0ex height1.0ex}
   
\vspace{10pt}

\noindent {\bf{Proof of Lemma \ref{lem:basicinequality}}}:  By Lemma \ref{observation}(b), (\ref{def:U}), the definition of $\tilde{\gamma}_k$ in (\ref{def:tildegamma}) and Proposition \ref{prop:technical1}, we have
\begin{eqnarray}
\lambda_k \phi(y_k) + \frac{1 - \gamma}{2} \| y_k - \tilde{x}_k \|^2 & \leq & \lambda_k \phi(y_k) + \frac{ 1 - U_k \lambda_k}{2} \| y_k - \tilde{x}_k \|^2 \nonumber \\
& = &  \lambda_k \tilde{\gamma}_k(y_k) + \frac{1}{2}(1 - \tau_k) \|y_k - \tilde{x}_k \|^2 \nonumber \\ 
%& = & \lambda_k \gamma_k(y_k) + \frac{1}{2}(1 - \tau_k) \|y_k - \tilde{x}_k \|^2 \\
& \leq & \lambda_k \gamma_k(y_k) + \frac{1}{2} \| y_k - \tilde{x}_k \|^2. \label{str:1}
\end{eqnarray}
Since $y_k$ is the optimal solution to the second minimization problem in (\ref{eq:minimizationequality}), by convexity of $\gamma_k$, (\ref{eq:all ak}) and (\ref{Akak}), the following holds for every $u \in \Omega$:
\begin{eqnarray}
& & A_k\left( \lambda_k \gamma_k (y_k) + \frac{1}{2} \|y_k - \tilde{x}_k \|^2 \right) \nonumber \\
& \leq & A_k \left( \lambda_k \gamma_k\left(\frac{A_{k-1} y_{k-1} + a_{k-1} x_k}{A_k} \right) + \frac{1}{2} \left\| \frac{A_{k-1} y_{k-1} + a_{k-1} x_k}{A_k} - \tilde{x}_k \right\|^2 \right) \nonumber \\
& \leq & \lambda_k A_{k-1} \gamma_k(y_{k-1}) + \lambda_k a_{k-1} \gamma_k(x_k) + \frac{A_k}{2} \left\| \frac{A_{k-1} y_{k-1} + a_{k-1} x_k}{A_k} - \tilde{x}_k \right\|^2 \nonumber \\
& = & \lambda_k A_{k-1} \gamma_k(y_{k-1}) + \lambda_k a_{k-1} \gamma_k(x_k) + \frac{1}{2} \| x_k - x_{k-1} \|^2 \nonumber \\
& \leq & \lambda_k A_{k-1} \gamma_k(y_{k-1}) + \lambda_k a_{k-1} \gamma_k(u) + \frac{1}{2} \|u - x_{k-1}\|^2 - \frac{a_{k-1}\tau_k + 1}{2} \| u - x_k \|^2, \label{str:2}
\end{eqnarray}
where the last inequality holds since $x_{k}$ is the optimal solution to the minimization problem (\ref{optimizationproblem}), and its objective function is $((a_{k-1} \tau_k + 1)/\lambda_k)$-strongly convex.  The result now follows by combining (\ref{str:1}) and (\ref{str:2}).
\hfill\hbox{\vrule width1.0ex height1.0ex}

%
%The first lemma in the appendix presents some estimates involving the two sequences of parameters $ \{a_k\} $ and $ \{A_k\} $ that are essential in the analysis of MOD-FISTA and ADAP-FISTA.
%
% \begin{lemma}\label{estimates}
%	For every $ k\ge 1 $, the sequences $ \{a_k\} $ and $ \{A_k\} $ defined in \eqref{ak} and \eqref{Aksequence}, respectively, satisfy
%	\[
%	A_k-A_0\ge\frac{k^2}{4}, \quad \sum_{i=0}^{k-1}A_{i+1}\ge\frac{k^3}{12}, \quad \frac{\sum_{i=0}^{k-1}a_i}{\sum_{i=0}^{k-1}A_{i+1}}\le \frac{4}{k}.
%	\]
%\end{lemma}
%\begin{proof}
%	The proof of the first inequality can be found in Lemma 7 of \cite{nesterov2012gradient}.
%	It is easy to see that the second inequality is a direct result of the first one.
%	Using relations \eqref{Aksequence} and \eqref{Ak+1ak}, the first inequality of the lemma and the Cauchy-Schwartz inequality, we conclude that the last inequality of the lemma
%	holds as follows:
%	\[
%	\frac{\sum_{i=0}^{k-1}a_i}{\sum_{i=0}^{k-1}A_{i+1}}=\frac{\sum_{i=0}^{k-1}a_i}{\sum_{i=0}^{k-1}a_i^2}
%	\le \frac{\sum_{i=0}^{k-1}a_i}{\frac1k (\sum_{i=0}^{k-1}a_i)^2}
%	= \frac{k}{A_k}
%	\le \frac{4}{k}.
%	\]
%\end{proof}

\end{document}